\documentclass[11pt]{amsart}
\usepackage{amsfonts}
\usepackage{mathrsfs}
\usepackage{amsmath}
\usepackage{amsmath, amsthm, amssymb}
\usepackage{color}

\newcommand{\T}{\theta}

\renewcommand{\(}{\left( }
\renewcommand{\)}{\right) }

\renewcommand{\theequation}{\theequation. \arabic{equation}}
\numberwithin{equation}{section}
\newtheorem{thm}{Theorem}[section]

\newtheorem{lem}[thm]{Lemma}

\newtheorem{prop}[thm]{Proposition}
\newtheorem{defn}[thm]{Definition}

\def\squarebox#1{\hbox to #1{\hfill\vbox to #1{\vfill}}}

\begin{document}
\title[ on the $q$-Hahn polynomials 
and the big $q$-Jacobi polynomials]
{A $q$-summation and the orthogonality relations for the $q$-Hahn polynomials 
and the big $q$-Jacobi polynomials}
\author{Zhi-Guo Liu}
\date{\today}
\address{ Department of Mathematics, East China Normal University, 500 Dongchuan Road,
Shanghai 200241, P. R. China} \email{zgliu@math.ecnu.edu.cn;
liuzg@hotmail.com}
\thanks{The  author was supported in part by
the National Science Foundation of China}
\begin{abstract}
Using a general $q$-summation formula,  we derive a generating function for the $q$-Hahn polynomials,
which is used  to give a complete proof of the orthogonality relation for the $q$-Hahn polynomials. 
A new proof of the orthogonality relation for the big $q$-Jacobi polynomials is also given. 
A simple evaluation of the Nassrallah-Rahman integral is derived by using this summation formula.
A new $q$-beta integral formula is established, which includes the Nassrallah-Rahman integral as a special
case. The $q$-summation formula also allows us  to recover several strange $q$-series identities.

\end{abstract}
\maketitle
\section{Introduction}
Throughout this paper we assume that $q$ is a complex number such that $|q|<1$.
For any complex number $a$, the $q$-shifted factorials $(a; q)_n$ are defined by
\begin{equation*}
(a; q)_0=1,\quad (a; q)_n=\prod_{k=0}^{n-1}(1-aq^k), \quad n=1, 2,\ldots, \text{or}~\infty.
\end{equation*}
For convenience, we also adopt the following compact notation for the multiple
$q$-shifted factorial:
\begin{equation*}
(a_1, a_2,...,a_m;q)_n=(a_1;q)_n(a_2;q)_n ... (a_m;q)_n.
\end{equation*}

The basic hypergeometric series or $q$-hypergeometric series
${_r\phi_s}(\cdot)$ are defined as
\begin{equation*}
{_r\phi_s} \left({{a_1, a_2, ..., a_{r}} \atop {b_1, b_2, ...,
b_s}} ;  q, z  \right) =\sum_{n=0}^\infty \frac{(a_1, a_2, ...,
a_{r};q)_n} {(q,  b_1, b_2, ..., b_s ;q)_n}\left((-1)^n q^{n(n-1)/2}\right)^{1+s-r} z^n.
\end{equation*}

For any function $f(x)$, the  $q$-derivative of $f(x)$
with respect to $x,$ is defined as
\begin{equation*}
\mathcal{D}_{q,x}\{f(x)\}=\frac{f(x)-f(qx)}{x},
\end{equation*}
and we further define  $\mathcal{D}_{q,x}^{0} \{f\}=f,$ and
for $n\ge 1$, $\mathcal{D}_{q, x}^n \{f\}=\mathcal{D}_{q, x}\{\mathcal{D}_{q, x}^{n-1}\{f\}\}.$

Using some basic properties of  the $q$-derivative, we \cite{Liu} proved the following $q$-expansion formula.
\begin{thm} {\rm (Liu)}\label{liuthm1} If $f(a)$ is an analytic function of $a$ near $a=0$, then,  we have
\begin{equation*}
f(a)=\sum_{n=0}^\infty \frac{(1-\alpha q^{2n})(\alpha q/a; q)_n a^n}{(q, a;
q)_n}\left[ \mathcal{D}^n_{q, x}\{f(x)(x; q)_{n-1}\} \right]_{x=\alpha q}.
\end{equation*}
\end{thm}
This theorem tells us that if $f(a)$ is analytic at $a=0$, then, it can be expanded uniquely in terms of $(\alpha q/a; q)_n a^n/(a; q)_n.$ 
Thus, if $f(a)$ has two series expansions in terms of $(\alpha q/a; q)_n a^n/(a; q)_n$, then,  
the corresponding coefficients of these two series must be equal.
This allows us to derive some combinatorial  identities by using the method of equating the coefficients.

Using Theorem \ref{liuthm1},  we \cite{Liu2013} established the following expansion theorem.
\begin{thm} {\rm (Liu)}\label{liuthm2}  If $f(x)$ is an analytic function near $x=0,$
then,  under suitable convergence conditions,  we have
\begin{align*}
&\frac{(\alpha q, \alpha ab/q; q)_\infty}
{(\alpha a, \alpha b; q)_\infty} f(\alpha a)\\
=\sum_{n=0}^\infty & \frac{(1-\alpha q^{2n}) (\alpha, q/a; q)_n (a/q)^n}
{(1-\alpha
)(q, \alpha a; q)_n}
\sum_{k=0}^n \frac{(q^{-n}, \alpha q^n; q)_k q^k}
{(q, \alpha b; q)_k}f(\alpha q^{k+1}).\nonumber
\end{align*}
\end{thm}
By choosing $ f(x)=\prod_{j=1}^m \frac{(b_jx/q; q)_\infty}{(c_jx/q; q)_\infty}$ in
Theorem~\ref{liuthm2}, we \cite{Liu2013qd} obtained  the following general $q$-summation formula.
\begin{thm} {\rm (Liu)} \label{liunewthmb} If
$
\max\{|\alpha a|, |\alpha b|, |\alpha b_1|, |\alpha ac_1/q|, \cdots |\alpha b_m|, |\alpha a c_m/q|\}<1,
$  then,  we have
 \begin{align*}
&  \frac{(\alpha q, \alpha ab/q; q)_\infty}
{(\alpha a, \alpha b; q)_\infty} \prod_{j=1}^m \frac{(\alpha a b_j/q, \alpha c_j; q)_\infty}
{(\alpha a c_j/q, \alpha b_j; q)_\infty}\\
&=\sum_{n=0}^\infty  \frac{(1-\alpha q^{2n}) (\alpha, q/a; q)_n (a/q)^n}
{(1-\alpha)(q, \alpha a; q)_n}
{_{m+2}\phi_{m+1}}\left( {{q^{-n}, \alpha q^n, \alpha c_1, \cdots,  \alpha c_m}
\atop{\alpha b, \alpha b_1, \cdots, \alpha b_m}}; q, q\right).
 \end{align*}
 \end{thm}
 This summation formula implies many nontrivial results in $q$-series as special cases.  For example,
 by setting $b_1=c$ and $c_1=bc/q$ in Theorem \ref{liunewthmb} and then using the $q$-Pfaff-Saalsch\"utz summation formula, we
 can obtain Rogers' $_6\phi_5$  summation formula, which is a $q$-analogue of Dougall's $_5F_4$
 summation formula.
\begin{thm} \label{rogersthm} For $|\alpha abc/q^2|<1,$ we have
\begin{align*}
{_6 \phi_5} \left({{\alpha, q\sqrt{\alpha}, -q\sqrt{\alpha}, q/a, q/b, q/c}
\atop{\sqrt{\alpha}, -\sqrt{\alpha},\alpha a, \alpha b, \alpha c}}; q, \frac{\alpha abc}{q^2}\right)
 =\frac{(\alpha q, \alpha ab/q, \alpha ac/q, \alpha bc/q; q)_\infty}
{(\alpha a, \alpha b, \alpha c, \alpha abc/q^2; q)_\infty}. \nonumber
\end{align*}
\end{thm}
In \cite{Liu2013qd}, Theorem \ref{liunewthmb} has  been used  to derive several important results in number theory, such as a
general formula for sums of any number of squares.

It is obvious that \cite[Theorem~1.2]{Liu2013} is  the case $m=2$ of Theorem \ref{liunewthmb}, which
has been used to give a surprising proof of the orthogonality relation for the Askey-Wilson polynomials.

 In this paper we continue to discuss some amazing application of Theorem \ref{liunewthmb}. In particular,
 this theorem is used to provide new proofs of the orthogonality relations for the $q$-Hahn polynomials 
 and the $q$-big Jacobi polynomials, and can also be used to recover some strange $q$-series identities.

 This paper is organized as follows.   In Section~2, we will use Theorem~\ref{liuthm1} to prove the
 following expansion theorem for the two-variable analytic functions.    
 This expansion theorem implies that if a function $f(a, b)$ can be expanded in terms of
\[
\frac{(\alpha q/a; q)_n (\alpha q/b; q)_m a^n b^m}
  {(q, a;q)_n (q, b; q)_m},
\]
 then, this expansion is unique. This fact enables us to use the method of equating the coefficients 
 to derive some identities.
 \begin{thm}\label{liudoublethm} If $f(a, b)$ is a two-variable analytic function at
  $(0, 0)\in \mathbb{C}^2$, then,  there exists a unique sequence $\{c_{n, m}\}$ independent of
  $a$ and $b$ such that
  \[
  f(a, b)=\sum_{n=0}^\infty \sum_{m=0}^\infty c_{n, m}
  \frac{(1-\alpha q^{2n})(1-\beta q^{2m})(\alpha q/a; q)_n (\alpha q/b; q)_m a^n b^m}
  {(q, a;q)_n (q, b; q)_m}.
  \]
 \end{thm}
 In Section~3, we use Theorem~\ref{liunewthmb} to give
 a complete proof of the orthogonality relation for the $q$-Hahn polynomials. 
 
 A new proof of the orthogonality relation for the big $q$-Jacobi polynomials 
 is derived in Section~4.  In Section~5, Theorem~\ref{liunewthmb} is used to give a new derivation of the Nassrallah-Rahman integral.
 Another proof of the Nassrallah-Rahman integral is given in Section~6.
The principal result in Section~7 is the following $q$-beta integral formula, which includes the Nassrallah-Rahman integral formula as a special case.
\begin{thm}\label{liunrintthm} Suppose that $q\alpha=a^2bcds $ and $\max\{|a|, |b|, |c|, |d|, |s|\}<1.$
Then we have the $q$-beta integral formula
\begin{align*}
&\int_{0}^{\pi} \frac{h(\cos 2\theta; 1)}{h(\cos \theta; a, b, c, d, s)}
{_3\phi_2}\({{ae^{i\theta}, ae^{i\theta}, \alpha uv/q}\atop{\alpha u, \alpha v}}; q, bcds\) d\theta\\
&=\frac{ 2\pi (abcd, abcs, abds, acds; q)_\infty}{(q, ab, ac, ad, as, bc, bd, bs, cd, cs, ds, q\alpha; q)_\infty}\\
&\quad \times
\sum_{n=0}^\infty \frac{(1-\alpha q^{2n}) (\alpha, q/u, q/v, ab, ac, ad, as; q)_n}
{(1-\alpha)(q, \alpha u, \alpha v, abcd, abcs, abds, acds; q)_n}(-\alpha^2 uv/a^2)^n q^{n(n-1)/2}.
\end{align*}
\end{thm}
 In Section~8, we use Theorem~\ref{liunewthmb} to recover some strange $q$-series identities.
\section{ the proof of Theorem~\ref{liudoublethm}}
To prove Theorem~\ref{liudoublethm}, we  need  the following
formula of F. H. Jackson \cite{Jackson1908}, which writes the
$n$th $q$-derivative of $f(x)$ in terms of $f(q^k x)$ for $k=0, 1, 2, \ldots, n.$
\begin{lem}{\rm(Jackson)}\label{Jaclemm}For any function $f(x),$ we have the identity
\begin{align*}
D_{q, x}^n \{f(x)\}=x^{-n} \sum_{k=0}^n \frac{(q^{-n}; q)_k}{(q; q)_k} q^k f(q^k x).
\end{align*}
\end{lem}
Now we begin to prove Theorem~\ref{liudoublethm} using Theorem~\ref{liuthm1} and Lemma~\ref{Jaclemm}.
\begin{proof} Since $f(a, b)$ is analytic at $(a, b)=(0, 0),$ $f(a, b)$ is analytic at $a=0,$ regarding
$b$ as constant. From Theorem~\ref{liuthm1}, we have that
\begin{equation}
f(a, b)=\sum_{n=0}^\infty \frac{(1-\alpha q^{2n})(\alpha q/a; q)_n a^n}{(q, a;
q)_n}\left[ \mathcal{D}^n_{q, x}\{f(x, b)(x; q)_{n-1}\} \right]_{x=\alpha q}.
\label{eqn:dliu1}
\end{equation}
Appealing to the Jackson formula in Lemma~\ref{Jaclemm}, we easily deduce that
\begin{align*}
&\left[ \mathcal{D}^n_{q, x}\{f(x, b)(x; q)_{n-1}\} \right]_{x=\alpha q}\\
&={(q \alpha)}^{-n} \sum_{k=0}^n \frac{(q^{-n}; q)_k}{(q; q)_k} q^k
(q^{k+1}\alpha; q)_{n-1} f(q^{k+1}\alpha, b).
\end{align*}
Since $f(a, b)$ is analytic at $(a, b)=(0, 0)$, we know that, for each $k\in \{0, 1, 2, \ldots, n\}$,  $f(q^{k+1}\alpha, b)$ is analytic at $b=0.$ It follows that  the left-hand
side of the above equation is also an analytic function of $b$ near $b=0.$ Using Theorem~\ref{liuthm1} again,
we find that
\begin{equation}
\left[ \mathcal{D}^n_{q, x}\{f(x, b)(x; q)_{n-1}\} \right]_{x=\alpha q}\label{eqn:dliu2}
\end{equation}
\[
=\sum_{m=0}^\infty \frac{(1-\beta q^{2m})(q\beta/b; q)_m b^m}{(q, b; q)_m}
\left[ \mathcal{D}^m_{q, y} \mathcal{D}^n_{q, x}\{f(x, y)(x; q)_{n-1}(y; q)_{m-1}\} \right]_{(x=\alpha q, y=\beta q)}.
\]
Letting $c_{n, m}=\left[ \mathcal{D}^m_{q, y} \mathcal{D}^n_{q, x}\{f(x, y)(x; q)_{n-1}(y; q)_{m-1}\} \right]_{(x=\alpha q, y=\beta q)}$,  it is obvious that $c_{m, n}$ are uniquely determined by $f(x, y).$
Combining (\ref{eqn:dliu1}) with (\ref{eqn:dliu2}), we complete the proof of Theorem~\ref{liudoublethm}.
\end{proof}

\section{ the orthogonality relation for the $q$-Hahn polynomials}
The $q$-Hahn polynomials are defined as (see, for example \cite{kalninsMiller})
\begin{equation}
H_n(a, b, c, d; z)=\frac{(ac, ad; q)_n}{a^n}{_3\phi_2}\left({{q^{-n}, abcdq^{n-1}, az}\atop{ac, ad}}; q, q\right).
\label{hpeqn1}
\end{equation}

For simplicity, in this section we denote $H_n(a, b, c, d; z):=H_n(z)$ and
\begin{equation}
A_n(a, b)=\frac{(1-abcdq^{2n-1})(abcdq^{-1}; q)_n a^n}{(1-abcdq^{-1})(q, ac, ad; q)_n}.
\label{hpeqn2}
\end{equation}
Using Theorems \ref{liunewthmb}, we can obtain the following generating function of $H_n(z).$
\begin{prop}\label{hahngenpp} For $\max\{|ac|, |ad|, |asz|, |abcds|\}<1,$ we have
\[
\sum_{n=0}^\infty \frac{s^n(s^{-1}; q)_n }{(abcds; q)_n}A_n(a, b)H_n(z)
=\frac{(abcd, acs, ads, az; q)_\infty}{(abcds, ac, ad, asz; q)_\infty}.
\]
\end{prop}
\begin{proof} Taking $m=1$ in Theorems \ref{liunewthmb} and then setting $a=qs$ and $b=t,$ we
deduce that
\begin{align*}
&\frac{(q\alpha, \alpha st, \alpha sb_1, \alpha c_1; q)_\infty}
{(q\alpha s, \alpha t, \alpha sc_1, \alpha b_1; q)_\infty}\\
&=\sum_{n=0}^\infty \frac{(1-\alpha q^{2n})(\alpha, s^{-1}; q)_n s^n}{(1-\alpha)(q, q\alpha s; q)_n}
{_3\phi_2}\left({{q^{-n}, \alpha q^n, \alpha c_1}\atop{\alpha t, \alpha b_1}}; q, q\right).
\end{align*}
Replacing $(t, b_1, c_1)$ by $(ac/\alpha, ad/\alpha, az/\alpha)$ in the above equation, we obtain
\begin{align*}
&\frac{(q\alpha, acs, ads, az; q)_\infty}
{(q\alpha s, ac, asz, ad; q)_\infty}\\
&=\sum_{n=0}^\infty \frac{(1-\alpha q^{2n})(\alpha, s^{-1}; q)_n s^n}{(1-\alpha)(q, q\alpha s; q)_n}
{_3\phi_2}\left({{q^{-n}, \alpha q^n, az}\atop{ac, ad}}; q, q\right).
\end{align*}
Putting $\alpha=abcdq^{-1}$ in the above equation and noticing the definition of
$A_n(a)$ and $H_n(z)$, we complete the proof of Proposition~\ref{hahngenpp}.
\end{proof}
It is well-known that $H_n(a, b, c, d; z)$ is symmetric in $a$ and $b$ (see, for example \cite[Eq. (2. 18)]{kalninsMiller},  \cite[Theorem~2]{LiuDMath}). Thus, by interchanging $a$ and $b$ in Proposition~\ref{hahngenpp} and replacing $s$ by $r$, we can obtain the
following proposition.
\begin{prop}\label{ahahngenpp} For $\max\{|bc|, |bd|, |brz|, |abcdr|\}<1,$ we have
\[
\sum_{m=0}^\infty \frac{r^m(r^{-1}; q)_m }{(abcdr; q)_m}A_m(b, a)H_m(z)
=\frac{(abcd, bcr, bdr, bz; q)_\infty}{(abcdr, bc, bd, brz; q)_\infty}.
\]
\end{prop}
\begin{prop}\label{qdougallpp} For $\max\{|q\alpha s|, |q\alpha r|\}<1,$ we have
\[
\sum_{n=0}^\infty \frac{(1-\alpha q^{2n})(\alpha, s^{-1}, r^{-1}; q)_n (-\alpha rs)^n q^{n(n+1)/2}}
{(1-\alpha)(q, q\alpha s, q\alpha r; q)_n}
=\frac{(q\alpha, q\alpha rs; q)_\infty}{(q\alpha s, q\alpha r; q)_\infty}.
\]
\end{prop}
\begin{proof} Setting $c=0$ in the Rogers $_6\phi_5$ summation in Theorem~\ref{rogersthm} and then replacing $(a, b)$ by $(qs, qr),$ we complete
the proof of the proposition.
\end{proof}
Askey and Roy \cite[Eq. (2. 8)]{AskeyRoy} used the Ramanujan $_1\psi_1$  summation to obtain the following interesting integral formula.
\begin{prop}\label{AskeyRoypp} For $\max\{|a|, |b|, |c|, |d|\}<1$ and $cd\rho \not=0,$ we have
\begin{align*}
&\frac{1}{2\pi}\int_{-\pi}^{\pi}\frac{(\rho e^{i\theta}/d, q d e^{-i\theta}/\rho, \rho c e^{-i\theta}, qe^{i\theta}/c\rho; q)_\infty}
{(ae^{i\theta}, be^{i\theta}, ce^{-i\theta}, de^{-i\theta}; q)_\infty} d\theta\\
&=\frac{(abcd, \rho, q/\rho, c\rho/d, qd/c\rho; q)_\infty}{(q, ac, ad, bc, bd; q)_\infty}.
\end{align*}
\end{prop}
Using the above four propositions we can derive the orthogonality relation for the $q$-Hahn polynomials.

For brevity, we now introduce $L_0, L_n$ and $K(\theta)$ as follows
\begin{align}
L_0&=\frac{(abcd, \rho, q/\rho, c\rho/d, qd/c\rho; q)_\infty}{(q, ac, ad, bc, bd; q)_\infty}, \nonumber\\
L_n&=\frac{(1-abcdq^{-1})(q, ac, ad, bc, bd; q)_n q^{n(n-1)}(-cd)^n}{(1-abcdq^{2n-1})(abcdq^{-1}; q)_n}L_0,\label{hpeqn3}\\
K(\theta)&=\frac{(\rho e^{i\theta}/d, q d e^{-i\theta}/\rho, \rho c e^{-i\theta}, qe^{i\theta}/c\rho; q)_\infty}
{(ae^{i\theta}, be^{i\theta}, ce^{-i\theta}, de^{-i\theta}; q)_\infty}.\nonumber
\end{align}
The orthogonality relation for the $q$-Hahn polynomials can be stated in the following theorem.
Kalnins and Miller \cite[Eq. (2. 8)]{kalninsMiller} proved the $m\not=n$ case of the theorem,
and also obtained a recurrence relation to evaluate the $m=n$ case. The value in this case is
given below, and it is what Kalnins and Miller could have stated using their recurrence
relation.

\begin{thm}\label{qhahnorthm} Let $H_n(z)$ be the $q$-Hahn polynomials and $L_n, K(\theta)$ be
defined by (\ref{hpeqn3}) . Then we have the orthogonality relation
\[
\frac{1}{2\pi}\int_{-\pi}^{\pi} K(\theta)H_n(e^{i\theta})H_m(e^{i\theta})d\theta=L_n \delta_{m, n}.
\]
\end{thm}
\begin{proof} Replacing $a$ by $as$ and $b$ by $br$ in the Askey-Roy integral, we easily find that
\begin{equation}
\frac{1}{2\pi}\int_{-\pi}^{\pi} K(\theta) \frac{(ae^{i\theta}, be^{i\theta}; q)}{(ase^{i\theta}, bre^{i\theta}; q)_\infty} d\theta
=\frac{(abcdrs, \rho, q/\rho, c\rho/d, qd/c\rho; q)_\infty}{(q, acs, ads, bcr, bdr; q)_\infty}.
\label{hahneqn1}
\end{equation}
Letting $z=e^{i\theta}$ in Propositions~\ref{hahngenpp} and \ref{ahahngenpp} and then multiplying the
two resulting equations together, we obtain
\begin{align*}
&\sum_{n, m=0}^\infty \frac{(s^{-1}; q)_n (r^{-1}; q)_m s^n r^m}{(abcds; q)_n (abcdr; q)_m}A_n(a, b)A_m(b, a) H_n(e^{i\theta})H_m(e^{i\theta})\\
&=\frac{(abcd, abcd, acs, ads, bcr, bdr, ae^{i\theta}, be^{i\theta}; q)_\infty}
{(abcds, abcdr, ad, ac, bc, bd, ase^{i\theta}, bre^{i\theta}; q)_\infty}.
\end{align*}
Substituting the above equation into (\ref{hahneqn1}), we easily obtain the series expansion
\begin{align*}
&\sum_{n, m=0}^\infty \frac{(s^{-1}; q)_n (r^{-1}; q)_m A_n(a, b)A_m(b, a) s^n r^m}{(abcds; q)_n (abcdr; q)_m} \frac{1}{2\pi}\int_{-\pi}^{\pi} K(\theta)H_n(e^{i\theta})H_m(e^{i\theta})d\theta\\
&=A_0 \frac{(abcd, abcdrs; q)_\infty}{(abcds, abcdr; q)_\infty}.
\end{align*}
Now we will use another method to expand the right-hand side member of the above equation in terms of 
\[
\frac{(s^{-1}; q)_n (r^{-1}; q)_m  s^n r^m}{(abcds; q)_n (abcdr; q)_m}.
\]  

In fact, by taking $\alpha=abcd/q$ in Proposition~\ref{qdougallpp}, we immediately find that
\begin{align*}
&\sum_{n=0}^\infty \frac{(1-abcd q^{2n-1})(abcdq^{-1}, s^{-1}, r^{-1}; q)_n (-abcd rs)^n q^{n(n-1)/2}}
{(1-abcdq^{-1})(q, abcds, abcdr; q)_n}\\
&=\frac{(abcd, abcdrs; q)_\infty}{(abcds, abcdr; q)_\infty}.
\end{align*}
Combining the above two equations, we are led to the series identity
\begin{align*}
&\sum_{n, m=0}^\infty \frac{(s^{-1}; q)_n (r^{-1}; q)_m A_n(a, b)A_m(b, a) s^n r^m}{(abcds; q)_n (abcdr; q)_m} \frac{1}{2\pi}\int_{-\pi}^{\pi} K(\theta)H_n(e^{i\theta})H_m(e^{i\theta})d\theta\\
&=A_0\sum_{n=0}^\infty \frac{(1-abcd q^{2n-1})(abcdq^{-1}, s^{-1}, r^{-1}; q)_n (-abcd rs)^n q^{n(n-1)/2}}
{(1-abcdq^{-1})(q, abcds, abcdr; q)_n}.
\end{align*}
Using Theorem~\ref{liudoublethm},  we can in the above equation equate the coefficients of
\[
\frac{(s^{-1}; q)_n (r^{-1}; q)_m  s^n r^m}{(abcds; q)_n (abcdr; q)_m}
\]
to arrive at the integral formula in Theorem~\ref{qhahnorthm}. This completes the
proof of the theorem.
\end{proof}
\section{ the orthogonality relation for the Big  $q$-Jacobi polynomials}
We begin this section with the following transformation formula for $q$-series.
\begin{prop}\label{BWWtrans}
For $\lambda=q\alpha^2/bcd$ and $|q\alpha/cd|<1,$ we have
\begin{align*}
&{_3\phi_2}\({{c, d, \alpha q/ab}\atop{\alpha q/a, \alpha q/b}}; q, \frac{q\alpha}{cd}\)
=\frac{(q\alpha/c, q\alpha/d, q\lambda/a; q)_\infty}
{(\alpha q/a, q\alpha/cd, q\lambda; q)_\infty}\\
&\qquad\times\sum_{n=0}^\infty \frac{(1-\lambda q^{2n})(\lambda, a, \lambda b/\alpha, \lambda c/\alpha, \lambda d/\alpha, q )_n}
{(1-\lambda)(q, q\lambda/a, q\alpha/b, q\alpha/c, q\alpha/d; q)_n}\(-\frac{q\alpha }{a}\)^n q^{n(n-1)/2}.
\end{align*}
\end{prop}
\begin{proof} Recall Watson's $q$-analogue of Whipple's theorem (see, for example \cite[p. 43, Eq. (2.5.1)]{Gas+Rah},
\cite[Proposition~10.3]{LiuRamanujanP})
\begin{align*}
&{_8\phi_7}\({{\alpha, q\alpha^{1/2}, -q\alpha^{1/2}, a, b, c, d, q^{-n}}
\atop{\alpha^{1/2}, -\alpha^{1/2}, q\alpha/a, q\alpha/b, q\alpha/c, q\alpha/d, \alpha q^{n+1} }}; q, \frac{\alpha^2 q^{2+n}}{abcd}\)\\
&=\frac{(q\alpha, q\alpha/cd; q)_n}{(q\alpha/c, q\alpha/d; q)_n}
{_4\phi_3}\({{q^{-n}, c, d, q\alpha/ab}\atop{q\alpha/a, q\alpha/b, cdq^{-n}/\alpha}}; q, q\).
\end{align*}
Setting $\lambda=q\alpha^2/bcd$ and using the transformation formula \cite[p. 49, Eq.(2.10.3)]{Gas+Rah}, we have that
\begin{align*}
&{_8\phi_7}\({{\alpha, q\alpha^{1/2}, -q\alpha^{1/2}, a, b, c, d, q^{-n}}
\atop{\alpha^{1/2}, -\alpha^{1/2}, q\alpha/a, q\alpha/b, q\alpha/c, q\alpha/d, \alpha q^{n+1} }}; q, \frac{\alpha^2 q^{2+n}}{abcd}\)\\
&=\frac{(q\alpha, q\lambda/a; q)_n}{(q\alpha/a, q\lambda; q)_n}
{_8\phi_7}\({{\lambda, q\lambda^{1/2}, -q\lambda^{1/2}, \lambda b/\alpha, \lambda c/\alpha, \lambda d/\alpha, a, q^{-n}}
\atop{\lambda^{1/2}, -\lambda^{1/2}, q\lambda/a, q\alpha/b, q\alpha/c, q\alpha/d, \lambda q^{n+1} }}; q, \frac{\alpha q^{1+n}}{a}\).
\end{align*}
Combining these two equations, we are led to the $q$-transformation formula
\begin{align*}
&{_4\phi_3}\({{q^{-n}, c, d, q\alpha/ab}\atop{q\alpha/a, q\alpha/b, cdq^{-n}/\alpha}}; q, q\)
=\frac{(q\alpha/c, q\alpha/d, q\lambda/a; q)_n}{(\alpha q/a, q\alpha/cd, q\lambda; q)_n}\\
&\qquad \times
{_8\phi_7}\({{\lambda, q\lambda^{1/2}, -q\lambda^{1/2}, \lambda b/\alpha, \lambda c/\alpha, \lambda d/\alpha, a, q^{-n}}
\atop{\lambda^{1/2}, -\lambda^{1/2}, q\lambda/a, q\alpha/b, q\alpha/c, q\alpha/d, \lambda q^{n+1} }}; q, \frac{\alpha q^{1+n}}{a}\).
\end{align*}
Letting $n\to \infty$ in the both sides of the above equation, we complete the proof of
Proposition~\ref{BWWtrans}.
\end{proof}
\begin{prop}\label{LBWW}If there are no zero factors in the denominator of the integral
and $\lambda=rhuv/q$, then,
we have
\begin{align*}
&\int_{u}^v \frac{(qx/u, qx/v, hx; q)_\infty}{(rx, sx, tx; q)_\infty} d_q x
=\frac{(1-q)v(q, u/v, qv/u, hu, hv, rsuv, rtuv; q)_\infty}
{(rhuv, ru, rv, su, sv, tu, tv; q)_\infty}\\
&\qquad \times
\sum_{n=0}^\infty \frac{(1-\lambda q^{2n})(\lambda, ru, rv, h/s, h/t; q)_n}
{(1-\lambda)(q, hu, hv, rsuv, rtuv; q)_n} (-stuv)^n q^{n(n-1)/2}.
\end{align*}
\end{prop}
\begin{proof}   In \cite[Theorem~9]{Liu2010}, we have proved  the $q$-integral formula
\begin{align*}
\int_{u}^v \frac{(qx/u, qx/v, hx; q)_\infty d_q x}{(rx, sx, tx; q)_\infty}
&=\frac{(1-q)v(q, u/v, qv/u, hv, stuv; q)_\infty}{(rv, su, sv, tu, tv; q)_\infty}\\
&\qquad \times {_3\phi_2}\({{h/r, sv, tv}\atop{stuv, hv}}; q, ru\).
\end{align*}
If we replace $(\alpha, a, b, c, d)$ by $(rstuv^2/q, rv, rstuv/h, sv, tv)$ in Proposition~\ref{BWWtrans},
then, we have $q\lambda=ruvh$ and
\begin{align*}
&{_3\phi_2}\({{h/r, sv, tv}\atop{stuv, hv}}; q, ru\)=\frac{(rsuv, rtuv, hu; q)_\infty}{(stuv, rhuv, ru; q)_\infty}\\
&\qquad \times
\sum_{n=0}^\infty \frac{(1-\lambda q^{2n})(\lambda, ru, rv, h/s, h/t; q)_n}
{(1-\lambda)(q, hu, hv, rsuv, rtuv; q)_n} (-stuv)^n q^{n(n-1)/2}.
\end{align*}
Combining the above two equations, we complete the proof of Proposition~\ref{LBWW}.
\end{proof}
The big $q$-Jacobi polynomials are defined as (see, for example \cite[p. 438]{KLS})
\begin{equation}
P_n(a, b, c; x)={_3\phi_2}\left({{q^{-n}, abq^{n+1}, x}\atop{qa, qc}}; q, q\right).
\label{bjpeqn1}
\end{equation}
For simplicity, in this section we use $P_n(x)$ to denote the big $q$-Jacobi polynomials.
Using Theorems \ref{liunewthmb}, we can obtain the following generating function of the
big $q$-Jacobi polynomials.
\begin{prop}\label{liuqjacobipp} For $\max\{|qa|, |qb|, |tx|, |qabt|\}<1,$ we have
\begin{equation*}
\frac{(qab, qat, qct, x; q)_\infty}
{(q^2abt, qa, qc, tx; q)_\infty}
=\sum_{n=0}^\infty  \frac{(1-abq^{2n+1}) (qab, 1/t; q)_n t^n}
{(q, q^2abt; q)_n}P_n(x).
\end{equation*}
\end{prop}
\begin{proof}
Setting  $c_1=\lambda \alpha^{-1} e^{i\theta},  c_2=\lambda \alpha^{-1} e^{-i\theta} $
and $b=0$ in Theorems \ref{liunewthmb}, we deduce that
\begin{align*}
&\frac{(\alpha q, \alpha ab_1/q, \alpha a b_2/q; q)_\infty}
{(\alpha a, \alpha b_1, \alpha b_2; q)_\infty}
\prod_{n=0}^\infty \frac{(1-2\lambda q^n \cos \theta+\lambda ^2 q^{2n})}
{(1-2\lambda a q^{n-1}\cos \theta+ \lambda^2 a^2 q^{2n-2})}\\
&=\sum_{n=0}^\infty  \frac{(1-\alpha q^{2n}) (\alpha, q/a; q)_n (a/q)^n}
{(1-\alpha)(q, \alpha a; q)_n}
{_{4}\phi_{3}}\left( {{q^{-n}, \alpha q^n, \lambda e^{i\theta},  \lambda e^{-i\theta}}
\atop{0, \alpha b_1, \alpha b_2}}; q, q\right).
\end{align*}
Taking $\cos \theta=x/(2 \lambda)$ in the above equation, letting $\lambda \to 0,$
and replacing  $a$ by $qt$, we find that
\begin{align*}
&\frac{(\alpha q, \alpha t b_1, \alpha t b_2, x; q)_\infty}
{(q\alpha t, \alpha b_1, \alpha b_2, tx; q)_\infty}\\
&=\sum_{n=0}^\infty  \frac{(1-\alpha q^{2n}) (\alpha, 1/t; q)_n t^n}
{(1-\alpha)(q, q\alpha t; q)_n}
{_{3}\phi_{2}}\left( {{q^{-n}, \alpha q^n, x}
\atop{\alpha b_1, \alpha b_2}}; q, q\right).
\end{align*}
Replacing $ \alpha b_1=qa, \alpha b_2=qc$ and $\alpha=qab$ in the above equation,
we complete the proof of Proposition~\ref{liuqjacobipp}.
\end{proof}
If the $q$-integral of the function $f(x)$ from $a$ to $b$ is defined as
\[
\int_{a}^b f(x) d_q x=(1-q)\sum_{n=0}^\infty [b f(bq^n)-af(aq^n)]q^n,
\]
then, the orthogonality relation for the big $q$-Jacobi polynomials can be stated in the following theorem 
(see, for example \cite[p. 182]{Gas+Rah}, \cite[p. 438]{KLS}).
\begin{thm}\label{qjacobiorth}  The orthogonality relation for the big $q$-Jacobi polynomials 
is 
\begin{align*}
&\int_{cq}^{aq}\frac{(x/a, x/c; q)_\infty}{(x, bx/c ; q)_\infty}P_m(x)P_n(x) d_q x
=aq(1-q)\frac{(q, abq^2, a/c, qc/a; q)_\infty}{(aq, bq, cq, abq/c; q)_\infty}\\
&\qquad \times \frac{(1-abq)(q, qb, abq/c; q)_n}{(1-abq^{2n+1})(aq, abq, cq; q)_n}
(-acq^2)^n q^{n(n-1)/2}\delta_{mn}.
\end{align*}
\end{thm}
\begin{proof} Replacing $t$ by $s$ in Proposition~\ref{liuqjacobipp}, we immediately deduce that
\begin{equation*}
\frac{(qab, qas, qcs, x; q)_\infty}
{(q^2abs, qa, qc, sx; q)_\infty}
=\sum_{m=0}^\infty  \frac{(1-abq^{2m+1}) (qab, 1/s; q)_m s^m}
{(q, q^2abs; q)_m}P_m(x).
\end{equation*}
If we multiply this equation with the equation in Proposition~\ref{liuqjacobipp} together,  we obtain
\begin{align*}
&\sum_{m, n=0}^\infty \frac{(1-abq^{2m+1})(1-abq^{2n+1})(qab; q)_m (qab; q)_n (1/s; q)_m (1/t; q)_ns^m t^n}
{(q; q)_m (q; q)_n (q^2abs; q)_m (q^2abt; q)_n}\\
&\qquad \times P_m(x)P_n(x)\\
&=\frac{(qab, x; q)_\infty ^2 (qas, qcs, qat, qct; q)_\infty}
{(qa, qc; q)^2_\infty (q^2abs, q^2abt, sx, tx; q)_\infty}.
\end{align*}
Multiplying the above equation by $(x/a, x/c; q)_\infty/(x, bx/c; q)_\infty$ and then taking the $q$-integral 
over $[cq, aq]$,  we find that
\begin{align*}
&\sum_{m, n=0}^\infty \frac{(1-abq^{2m+1})(1-abq^{2n+1})(qab; q)_m (qab; q)_n (1/s; q)_m (1/t; q)_ns^m t^n}
{(q; q)_m (q; q)_n (q^2abs; q)_m (q^2abt; q)_n}\\
&\qquad \times    \int_{cq}^{aq}\frac{(x/a, x/c; q)_\infty}{(x, bx/c ; q)_\infty}P_m(x)P_n(x) d_q x  \\
&=\frac{(qab; q)_\infty ^2 (qas, qcs, qat, qct; q)_\infty}
{(qa, qc; q)^2_\infty (q^2abs, q^2abt; q)_\infty}
\int_{cq}^{aq} \frac{(x/a, x/c, x; q)_\infty d_q x}{(bx/c, sx, tx; q)_\infty}.
\end{align*}
Setting $(h, r, u, v)=(1, b/c, qc, qa)$ in Proposition~\ref{LBWW} and simplifying, we obtain
\begin{align*}
&\int_{cq}^{aq} \frac{(x/a, x/c, x; q)_\infty d_q x}{(bx/c, sx, tx; q)_\infty}
=\frac{aq(1-q)(q, c/a, qa/c, qa, qc, absq^2, abtq^2; q)_\infty}
{(abq^2, qb, abq/c, aqs, aqt, cqs, cqt; q)_\infty}\\
&\qquad 
\sum_{n=0}^\infty \frac{(1-abq^{2n+1})(abq, bq, abq/c, 1/s, 1/t; q)_n}
{(1-abq)(q, aq, cq, absq^2, abtq^2; q)_n} (-acst q^2)^n q^{n(n-1)/2}.
\end{align*}
Combining the above equations, we deduce that
\begin{align*}
&\sum_{m, n=0}^\infty \frac{(1-abq^{2m+1})(1-abq^{2n+1})(qab; q)_m (qab; q)_n (1/s; q)_m (1/t; q)_ns^m t^n}
{(q; q)_m (q; q)_n (q^2abs; q)_m (q^2abt; q)_n}\\
&\qquad \times    \int_{cq}^{aq}\frac{(x/a, x/c; q)_\infty}{(x, bx/c ; q)_\infty}P_m(x)P_n(x) d_q x  \\
&=\frac{aq(1-q)(q, qab, c/a, qa/c; q)_\infty}
{(qa, qb, qc, qab/c; q)_\infty }\\
&\qquad \times \sum_{n=0}^\infty \frac{(1-abq^{2n+1})(abq, bq, abq/c, 1/s, 1/t; q)_n}
{(q, aq, cq, absq^2, abtq^2; q)_n} (-acst q^2)^n q^{n(n-1)/2}.
\end{align*}
Using Theorem~\ref{liudoublethm}, we can obtain the orthogonality relation for the big $q$-Jacobi polynomials 
 by  equating the coefficients of 
\[
\frac{(1/s; q)_m (1/t; q)_ns^m t^n}{(q^2abs; q)_m (q^2abt; q)_n}
\]
in the above equation. This completes the proof of Theorem~\ref{qjacobiorth}.
\end{proof}
\section{on the Nassrallah-Rahman integral}
Analogous to the hypergeometric case, we call the $q$-hypergeometric series
\[
{_{r+1}\phi_r}\({{a_1, a_2, \ldots, a_{r+1}}\atop{b_1, \ldots, b_r}}; q, z\)
\]
well-poised if the parameters satisfy the relations
$
qa_1=a_2b_1=a_3b_2=\cdots=a_{r+1}b_r;
$
very-well-poised if, in addition, $a_2=q \sqrt{a_1}, a_3=-q \sqrt{a_1}.$

For simplicity, we sometimes use $_{r+1}W_r (a_1; a_4, a_5, \ldots, a_{r+1}; q, z)$ to denote
 \[
{_{r+1}\phi_r}\({{a_1, q\sqrt{a_1}, -q\sqrt{a_1}, a_4, \ldots, a_{r+1}}\atop{\sqrt{a_1}, -\sqrt{a_1}, qa_1/a_4, \ldots, qa_1/a_{r+1}}}; q, z\)
\]

\begin{defn}\label{defn1}
For$\ x=\cos \T$, we define the notation $h(x; a)$ and
 $ h(x; a_1, a_2, \ldots, a_m)$  as follows
 \begin{align*}
&h(x; a)=(ae^{i\T}, ae^{-i\T}; q)_\infty=\prod_{k=0}^\infty (1-2q^k ax+q^{2k}a^2) \\
&h(x; a_1, a_2, \ldots, a_m)=h(x; a_1)h(x; a_2)\cdots h(x; a_m).
 \end{align*}
\end{defn}
The following important integral evaluation is due to Askey and Wilson,
which can be used to obtain an elegant proof of the orthogonality relation for these Askey-Wilson polynomials.
\begin{thm}\label{thmasw} {\rm(Askey--Wilson)} With $h(x; a)$ being defined in Definition \ref{defn1} and
$\max\{|a|, |b|, |c|, |d|\}<1$, we have
\begin{equation}
I(a, b, c, d):= \int_{0}^{\pi} \frac{h(\cos 2\T; 1)d\T}{h(\cos \T; a, b, c,  d)}
=\frac{2\pi (abcd; q)_\infty}{(q, ab, ac, ad, bc, bd, cd;q)_\infty}. \label{asweqn1}
\end{equation}
\end{thm}

For $x=\cos \theta,$ the Askey-Wilson polynomials $p_n(a, b, c, d; \cos \theta)$
are defined as \cite{Ask+Wil}, \cite[p. 188]{Gas+Rah}
\begin{align}
(ab, ac, ad)_n a^{-n}{_4}\phi{_3} \left({q^{-n}, abcdq^{n-1}, a
e^{i\theta}, ae^{-i\theta}\atop{ab, ac, ad}}; q, q\right).\label{askey1}
\end{align}
Using Theorem~\ref{liunewthmb}, we can obtain the following generating function for the Askey-Wilson
polynomials.
\begin{prop}\label{awgenpp} For $\max\{|abcds|, |ab|, |ac|, |ad|, |sae^{i\theta}|, |sae^{-i\theta}|\}<1,$
we have
\begin{align*}
&\frac{(abcd, abs, acs, ads, ae^{i\theta}, ae^{-i\theta}; q)_\infty}
{(abcds, ab, ac, ad, sae^{i\theta}, sae^{-i\theta}; q)_\infty}\\
&=\sum_{n=0}^\infty \frac{(1-abcdq^{2n-1})(abcdq^{-1}, s^{-1}; q)_n (sa)^n}
{(1-abcdq^{-1})(q, ab, ac, ad, abcds; q)_n}p_n(a, b, c, d; \cos \theta).
\end{align*}
\end{prop}
\begin{proof}
Taking $m=2$ in Theorem~\ref{liunewthmb} and then setting $a=qs$ and $b=t,$ we conclude that
\begin{align*}
&\frac{(q\alpha, \alpha st, \alpha b_1s, \alpha b_2s, \alpha c_1, \alpha c_2; q)_\infty}
{(q\alpha s, \alpha t, \alpha c_1 s, \alpha  c_2s, \alpha b_1, \alpha b_2; q)_\infty}\\
&=\sum_{n=0}^\infty \frac{(1-\alpha q^{2n})(\alpha, s^{-1}; q)_n s^n}{(1-\alpha)(q, q\alpha s; q)_n}
{_4\phi_3}\left({{q^{-n}, \alpha q^n, \alpha c_1, \alpha c_2}\atop{\alpha t, \alpha b_1, \alpha b_2}}; q, q\right).
\end{align*}
Replacing $(\alpha t, \alpha b_1, \alpha b_2, \alpha c_1, \alpha c_2 )$ by $(ab, ac, ad, a
e^{i\theta}, ae^{-i\theta})$ in the above equation and then taking $\alpha=abcdq^{-1}$, we
complete the proof of the theorem.
\end{proof}
Nassrallah and Rahman \cite{NassrallahRahman} used the integral representation of the sum of
two non-terminating $_3\phi_2$ series and the Askey-Wilson integral to find the following
$q$-beta integral formula. In this section, we will use Proposition~\ref{awgenpp} to give
a new proof of the Nassrallah-Rahman integral formula \cite[Eq. (6.3.7)]{Gas+Rah}.
\begin{thm}\label{nrintthma} {\rm (Nassrallah-Rahman)} For $\max\{|a|, |b|, |c|, |d|, |s|\}<1,$
we have
\begin{align*}
&\int_{0}^{\pi}
\frac{h(\cos 2\theta; 1)h(\cos \theta; r) d\theta}
{h(\cos \theta; a, b, c, d, s)} \\
&=\frac{2\pi(r/s, rs, abcs, bcds, acds, abds; q)_\infty}
{(q, ab, ac, ad, as, bc, bd, bs, cd, cs, ds, abcds^2 ; q)_\infty}\\
&\quad \times {_8W_7}(abcds^2/q; as, bs, cs, ds, abcds/r; q, r/s).
\nonumber
\end{align*}
\end{thm}
\begin{proof} Replacing $a$ by $r$ in Proposition~\ref{awgenpp}, we immediately have
\begin{align*}
&\frac{(rbcd, rbs, rcs, rds; q)_\infty h(\cos \theta; r)}
{(rbcds, rb, rc, rd; q)_\infty h(\cos \theta; rs)}\\
&=\sum_{n=0}^\infty \frac{(1-rbcdq^{2n-1})(rbcdq^{-1}, s^{-1}; q)_n (rs)^n}
{(1-rbcdq^{-1})(q, rb, rc, rd, rbcds; q)_n}p_n(r, b, c, d; \cos \theta).
\end{align*}
It is well-know that the Askey-Wilson polynomials $p_n(r, b, c, d; \cos \theta)$ is symmetric in
$r, b, c$ and $d$ (see, for example \cite[Corollary~4]{Liu2011}).  Thus,  we have
 \begin{align*}
&\frac{(rbcd, rbs, rcs, rds; q)_\infty h(\cos \theta; r)}
{(rbcds, rb, rc, rd; q)_\infty h(\cos \theta; rs)}\\
&=\sum_{n=0}^\infty \frac{(1-rbcdq^{2n-1})(rbcdq^{-1}, s^{-1}; q)_n (rs)^n}
{(1-rbcdq^{-1})(q, rb, rc, rd, rbcds; q)_n}p_n( d, b, c, r; \cos \theta)\\
&=\sum_{n=0}^\infty \frac{(1-rbcdq^{2n-1})(bd, cd, rbcdq^{-1}, s^{-1}; q)_n (rs/d)^n}
{(1-rbcdq^{-1})(q, rb, rc, rbcds; q)_n} \\
&\qquad\times{_4}\phi{_3} \left({q^{-n}, rbcdq^{n-1}, d
e^{i\theta}, de^{-i\theta}\atop{bd, cd, rd}}; q, q\right).
\end{align*}
Multiplying the above equation by $h(\cos 2\theta; 1)/h(\cos \theta; a, b, c, d), $
and then taking the definite integral over $0\le \theta \le \pi,$ we find that
\begin{align*}
&\frac{(rbcd, rbs, rcs, rds; q)_\infty }{(rbcds, rb, rc, rd; q)_\infty }
\int_{0}^{\pi}\frac{h(\cos 2\theta; 1)h(\cos \theta; r) d\theta}{h(\cos \theta; a, b, c, d, sr)}\\
&=\sum_{n=0}^\infty \frac{(1-rbcdq^{2n-1})(bd, cd, rbcdq^{-1}, s^{-1}; q)_n (rs/d)^n}
{(1-rbcdq^{-1})(q, rb, rc, rbcds; q)_n}\\
&\qquad \times \sum_{k=0}^n \frac{(q^{-n}, rbcdq^{n-1}; q)_k q^k}{(q, bd, cd, rd)_k}
\int_{0}^{\pi}\frac{h(\cos 2\theta; 1) d\theta}{h(\cos \theta; a, b, c, dq^k)}.
\end{align*}
With the help of the Askey-Wilson integral in Theorem~\ref{thmasw}, we  find that
\[
\int_{0}^{\pi}\frac{h(\cos 2\theta; 1) d\theta}{h(\cos \theta; a, b, c, dq^k)}
=\frac{2\pi  (ad, bd, cd; q)_k (abcd; q)_\infty}{(abcd; q)_k(q, ab, ac, ad, bc, bd, cd; q)_\infty}.
\]
It follows that
\begin{align*}
&\sum_{k=0}^n \frac{(q^{-n}, rbcdq^{n-1}; q)_k q^k}{(q, bd, cd, rd)_k}
\int_{0}^{\pi}\frac{h(\cos 2\theta; 1) d\theta}{h(\cos \theta; a, b, c, dq^k)}\\
&=\frac{2\pi (abcd; q)_\infty}{(q, ab, ac, ad, bc, bd, cd; q)_\infty}
{_3}\phi{_2} \left({{q^{-n}, rbcdq^{n-1}, ad}\atop{abcd, rd}}; q, q\right)
\end{align*}
We can apply the $q$-Pfaff-Saalsch\"{u}tz formula to sum the ${_3\phi_2}$
series on the right hand side of the above equation to obtain
\[
{_3}\phi{_2} \left({{q^{-n}, rbcdq^{n-1}, ad}\atop{abcd, rd}}; q, q\right)
=\frac{(bc, r/a; q)_n (ad)^n}{(rd, abcd; q)_n}.
\]
Combining the above equations, we finally conclude that
\begin{align*}
&\int_{0}^{\pi}\frac{h(\cos 2\theta; 1)h(\cos \theta; r) d\theta}{h(\cos \theta; a, b, c, d, sr)}\\
&=\frac {2\pi(abcd, rbcds, rb, rc, rd; q)_\infty }{(q, ab, ac, ad, bc, bd, cd, rbcd, rbs, rcs, rds; q)_\infty }\\
&\times{_8W_7}(rbcdq^{-1}; bc, bd, cd, s^{-1}, r/a; q, ars).
\end{align*}
Replacing $s$ by $s/r$ in the above equation and then interchanging $a$ and $d$,  we deduce that
\begin{align}
&\int_{0}^{\pi}\frac{h(\cos 2\theta; 1)h(\cos \theta; r) d\theta}{h(\cos \theta; a, b, c, d, s)}\nonumber\\
&=\frac {2\pi(abcd, abcs, ra,  rb, rc; q)_\infty }{(q, ab, ac, ad, bc, bd, cd, rabc, as, bs, cs; q)_\infty }\label{nreqn}\\
&\qquad \times {_8W_7}(rabcq^{-1}; r/s,  ab, ac, bc, r/d; q, ds).\nonumber
\end{align}
Applying the transformation formula for $_8\phi_7$ in \cite[III. 24]{Gas+Rah} to the right-hand of the above equation,
we complete the proof of Theorem~\ref{nrintthma}.
\end{proof}
Taking $r=0$ in (\ref{nreqn}), we immediately find,   for $\max\{|a|, |b|, |c|, |d|, |s|\}<1,$ that
\begin{align}
&\int_{0}^{\pi}\frac{h(\cos 2\theta; 1) d\theta}{h(\cos \theta; a, b, c, d, s)}\label{liubeqn}\\
&=\frac {2\pi(abcd, abcs; q)_\infty }{(q, ab, ac, ad, bc, bd, cd, as, bs, cs; q)_\infty }
 {_3\phi_2}\left({{ab, ac, bc}\atop{abcd, abcs}}; q, ds\right).\nonumber
\end{align}
\section{another proof of the Nassrallah-Rahman integral formula}

The Al-Salam and Verma $q$-integral formula \cite[Eq. (1.3)]{SalamVerma} can be stated in the following proposition.
\begin{prop}\label{salamverma}  If there are no zero factors in the denominator of the integral, then,
we have
\[
\int_{d}^s \frac{(qx/d, qx/s, abcdsx; q)_\infty}{(ax, bx, cx; q)_\infty} d_q x=
\frac{(1-q)s (q, d/s, qs/d, abds, acds,  bcds; q)}{(ad, as, bd, bs, cd, cs; q)_\infty}.
\]
\end{prop}
The following $q$-integral formula is a special case of \cite[Eq. (2.10.19)]{Gas+Rah}.
 Now we will use Theorem~\ref{nrintthma} to give a derivation of this $q$-integral formula.
\begin{thm}\label{qbaileythm}
If there are no zero factors in the denominator of the integral and $|r/s|<1$, then,
we have
\begin{align*}
&\int_{d}^s \frac{(abcx, qx/d, qx/s, rx; q)_\infty d_qx}{( ax, bx, cx, rx/ds; q)_\infty}\\
&=\frac{(1-q)s(q, d/s, qs/d, rs, abcs, acds, abds, bcds; q)_\infty }
{(r/d, ad, bd, cd, as, bs,cs,  abcds^2; q)_\infty }\\
&\qquad \times {_8W_7}(abcds^2/q; as, bs, cs, ds, abcds/r; q, r/s).
\end{align*}
\end{thm}
\begin{proof} Letting $d=x$ in the Askey-Wilson integral in Theorem~\ref{thmasw} and then multiplying both sides
of the resulting equation by $(qx/d, qx/s, drsx; q)_\infty/(rx; q)_\infty $, we obtain
\[
\int_{0}^{\pi} \frac{h(\cos 2\theta; 1)(qx/d, qx/s, drsx; q)_\infty d\theta}{h(\cos \theta; a, b, c, x)(rx; q)_\infty}
=\frac{2\pi (abcx, qx/d, qx/s, drsx; q)_\infty}{(q, ab, ac, bc, ax, bx, cx, rx; q)_\infty}.
\]
Taking the $q$-integral over $d\le x\le s$ in the both side of the above equation, we obtain
\begin{align*}
&\int_{0}^{\pi} \frac{h(\cos 2\theta; 1) d\theta}{h(\cos \theta; a, b, c)}\int_{d}^s \frac{(qx/d, qx/s, drsx; q)_\infty d_qx}
{(xe^{i\theta}, xe^{-i\theta}, rx; q)_\infty}\\
&=\frac{2\pi}{(q, ab, ac, bc; q)_\infty}\int_{d}^s \frac{(abcx, qx/d, qx/s, drsx; q)_\infty d_qx}{( ax, bx, cx, rx; q)_\infty}.
\end{align*}
Using the Al-Salam and Verma $q$-integral formula, we immediately find that
\[
\int_{d}^s \frac{(qx/d, qx/s, drsx; q)_\infty d_qx}
{(xe^{i\theta}, xe^{-i\theta}, rx; q)_\infty}
=\frac{(1-q)s(q, d/s, qs/d, ds; q)_\infty h(\cos \theta; drs)}
{(rd, rs; q)_\infty h(\cos \theta; d, s)}
\]
Combining the above two equations, we arrive at
\begin{align*}
&\int_{d}^s \frac{(abcx, qx/d, qx/s, drsx; q)_\infty d_qx}{( ax, bx, cx, rx; q)_\infty}\\
&=\frac{(1-q)s(q, q, ab, ac, bc,  d/s, qs/d, ds; q)_\infty }
{2\pi(rd, rs; q)_\infty } \int_{0}^{\pi} \frac{h(\cos 2\theta; 1) h(\cos \theta; drs) d\theta}{h(\cos \theta; a, b, c, d, s)}.
\end{align*}
Replacing $r$ by $r/ds$ in the above equation, we find that
\begin{equation}
\int_{d}^s \frac{(abcx, qx/d, qx/s, rx; q)_\infty d_qx}{( ax, bx, cx, rx/ds; q)_\infty}\label{qbaileyeqn1}
\end{equation}
\[
=\frac{(1-q)s(q, q, ab, ac, bc,  d/s, qs/d, ds; q)_\infty }
{2\pi(r/d, r/s; q)_\infty } \int_{0}^{\pi} \frac{h(\cos 2\theta; 1) h(\cos \theta; r) d\theta}{h(\cos \theta; a, b, c, d, s)}.
\]
Applying Theorem~\ref{nrintthma} to the above equation, we complete the proof of Theorem~\ref{qbaileythm}.
\end{proof}
Next we will use (\ref{qbaileyeqn1}) to give another proof of the Nassrallah-Rahman integral formula.
\begin{proof}
If we choose $r=abcds$ in (\ref{qbaileyeqn1}),  we immediately deduce that
\begin{align*}
&\int_{d}^s \frac{(abcdsx, qx/d, qx/s; q)_\infty d_qx}{( ax, bx, cx; q)_\infty}\\
&=\frac{(1-q)s(q, q, ab, ac, bc,  d/s, qs/d, ds; q)_\infty }
{2\pi(abcs, abcd; q)_\infty } \int_{0}^{\pi} \frac{h(\cos 2\theta; 1) h(\cos \theta; abcds) d\theta}{h(\cos \theta; a, b, c, d, s)}.
\end{align*}
Applying the Al-Salam and Verma $q$-integral formula to the left-hand of the above equation and simplifying, we
find, for $\max\{|a|, |b|, |c|, |d|, |s|\}<1,$ that \cite[Eq. (6.4.1)]{Gas+Rah}
\begin{align}
&\int_{0}^{\pi} \frac{h(\cos 2\theta; 1) h(\cos \theta; abcds) d\theta}{h(\cos \theta; a, b, c, d, s)} \label{nhinteqn}\\
&=\frac{2\pi (abcd, abcs, abds, acds, bcds; q)_\infty}{(q, ab, ac, ad, as, bc, bd, bs, cd, cs, ds; q)_\infty}.
\nonumber
\end{align}
Using  $T(\theta)$ to denote the integrand of the above integral
and the value of this integral by $I, $ then,  we have
\[
\int_{0}^{\pi} T(\theta){d\theta}
=I.
\]
Replacing $s$ by $sq^n$ in the above equation and simplifying, we easily obtain
\[
\int_{0}^{\pi}  \frac{T(\theta)(se^{i\theta}, se^{-i\theta}; q)_n}
{(abcds e^{i\theta}, abcds e^{-i\theta} ; q)_n}{d\theta}
=\frac{(as, bs, cs, ds; q)_n }{(abcs, abds, acds, bcds; q)_n}I.
\]
If we multiply both sides of the above equation by the following factor:
\[
\frac{(1-abcds^2 q^{2n-1})(abcds^2/q, abcds/r; q)_n (r/s)^n}{(1-abcds^2/q)(q, rs, r/s; q)_n},
\]
and then summing the resulting the equation over, $0\le n \le \infty,$ we find that
\begin{align*}
&\int_{0}^{\pi} T(\theta) {_6W_5}(abcds^2/q; abcds/r, se^{i\theta}, se^{-i\theta}; q, r/s)
{d\theta}\\
&= {_8W_7}(abcds^2/q; as, bs, cs, ds, abcds/r; q, r/s)I.
\end{align*}
Using the $q$-Dougall summation, we immediately find that
\[
{_6W_5}(abcds^2/q; abcds/r, se^{i\theta}, se^{-i\theta}; q, r/s)
=\frac{(abcds^2, abcd; q)_\infty h(\cos \theta; r)}
{(rs, r/s; q)_\infty h(\cos \theta; abcds)}.
\]
Combining the above two equations, we complete the proof of the theorem.
\end{proof}
\section{A new $q$-beta integral formula}
We first recall the following well-known $q$-formula (see, for example \cite[p. 62]{Gas+Rah}, \cite[Theorem~1.8]{Liu2013}).
\begin{prop}\label{qtransfliu}For $|\alpha xy/q|<1$, we have the $q$-transformation formula
\begin{align*}
&\frac{(\alpha q, \alpha xy/q; q)_\infty}
{(\alpha x, \alpha y; q)_\infty} {_3\phi_2} \left({{q/x, q/y, \alpha uv/q} \atop {\alpha u, \alpha v}} ;  q, \frac{\alpha xy}{q} \right) \\
&=\sum_{n=0}^\infty \frac{(1-\alpha q^{2n}) (\alpha, q/x, q/y, q/u, q/v; q)_n (-\alpha^2 xyuv/q^2)^n q^{n(n-1)/2}}
{(1-\alpha)(q, \alpha x, \alpha y, \alpha u, \alpha v; q)_n}.
\end{align*}
\end{prop}
Now we begin to prove Theorem~\ref{liunrintthm} by using the above proposition and the $q$-beta integral formula in (\ref{nhinteqn}).
\begin{proof} Setting $x=(q/a) e^{i\theta}, y=(q/a) e^{-i\theta}$ and $\alpha=a^2bcds/q$ in Proposition~\ref{qtransfliu}, we
obtain
\begin{align*}
&\frac{(q\alpha, bcds; q)_\infty}
{h(\cos \theta; abcds)} {_3\phi_2} \left({{ae^{i\theta}, ae^{i\theta}, \alpha uv/q} \atop {\alpha u, \alpha v}} ;  q, bcds \right) \\
&=\sum_{n=0}^\infty \frac{(1-\alpha q^{2n}) (\alpha, ae^{i\theta}, ae^{i\theta} , q/u, q/v; q)_n (-\alpha^2 uv/a^2)^n q^{n(n-1)/2}}
{(1-\alpha)(q , abcdse^{i\theta}, abcds e^{i\theta}, \alpha u, \alpha v; q)_n}.
\end{align*}
If we multiply both sides of the above equation by the factor
\[
\frac{h(\cos 2\theta; 1)h(\cos \theta; abcds)}{h(\cos \theta; a, b, c, d, s)},
\]
and then take the definite integral over $0\le \theta \le \pi$ in the resulting equation,  we deduce that
\begin{align*}
&\int_{0}^{\pi} \frac{h(\cos 2\theta; 1)}{h(\cos \theta; a, b, c, d, s)}
{_3\phi_2}\({{ae^{i\theta}, ae^{i\theta}, \alpha uv/q}\atop{\alpha u, \alpha v}}; q, bcds\) d\theta\\
&=\frac{1}{(q\alpha, bcds; q)_\infty}
\sum_{n=0}^\infty \frac{(1-\alpha q^{2n}) (\alpha, q/u, q/v; q)_n}
{(1-\alpha)(q, \alpha u, \alpha v; q)_n}
\left(-\alpha^2 uv/a^2 \right)^n q^{n(n-1)/2}\\
&\quad \times
\int_{0}^{\pi} \frac{h(\cos 2\theta; 1) h(\cos \theta; abcdsq^n) d\theta}{h(\cos \theta; aq^n, b, c, d, s)}.
\end{align*}
If $a$ is replace by $aq^n$ in (\ref{nhinteqn}), then,  we immediately conclude  that
\begin{align*}
&\int_{0}^{\pi} \frac{h(\cos 2\theta; 1) h(\cos \theta; abcdsq^n) d\theta}{h(\cos \theta; aq^n, b, c, d, s)} \\
&=\frac{2\pi (abcd, abcs, abds, acds, bcds; q)_\infty (ab, ac, ad, as; q)_n}
{(q, ab, ac, ad, as, bc, bd, bs, cd, cs, ds; q)_\infty (abcd, abcs, abds, acds; q)_n}.
\end{align*}
Combining the above two equations, we complete the proof of Theorem~\ref{liunrintthm}.
\end{proof}
When $s=0$,  it is obvious that Theorem~\ref{liunrintthm} immediately becomes the Askey-Wilson integral.

When $u=q,$ the series in the right-hand side of the equation in Theorem~\ref{liunrintthm} immediately reduces to $1,$
and the $_3\phi_2$ series becomes a $_2\phi_1$ series which can be summed by the $q$-Gauss summation formula,
\[
{_3\phi_2}\({{ae^{i\theta}, ae^{i\theta}}\atop{q\alpha}}; q, bcds\)
=\frac{h(\cos \theta; abcds)}{(q\alpha, bcds; q)_\infty}.
\]
Hence, in this case, the integral formula in Theorem~\ref{liunrintthm} becomes the $q$-integral formula
(\ref{nhinteqn}).

Letting $v \to \infty$ in Theorem~\ref{liunrintthm} and by a direct computation, we easily deduce that
\begin{align*}
&\int_{0}^{\pi} \frac{h(\cos 2\theta; 1)}{h(\cos \theta; a, b, c, d, s)}
{_2\phi_1}\({{ae^{i\theta}, ae^{i\theta}}\atop{\alpha u}}; q, \frac{\alpha u}{a^2}\) d\theta\\
&=\frac{ 2\pi (abcd, abcs, abds, acds; q)_\infty}{(q, ab, ac, ad, as, bc, bd, bs, cd, cs, ds, q\alpha; q)_\infty}\\
&\quad \times
{_8W_7}(\alpha; q/u, ab, ac, ad, as; q,  \alpha uv/a^2).
\end{align*}
The $_2\phi_1$ series in the above equation can be summed by the $q$-Gauss summation,
\[
{_2\phi_1}\({{ae^{i\theta}, ae^{i\theta}}\atop{\alpha u}}; q, \frac{\alpha u}{a^2}\)=\frac{h(\cos \theta; \alpha u/ a)}{(\alpha u, \alpha u/a^2; q)_\infty}.
\]
Combining the above two equations, we are led to the $q$-beta integral formula
\begin{align*}
&\int_{0}^{\pi} \frac{h(\cos 2\theta; 1)h(\cos \theta; \alpha u/ a)}{h(\cos \theta; a, b, c, d, s)} d\theta\\
&=\frac{ 2\pi (abcd, abcs, abds, acds, \alpha u, \alpha u/a^2; q)_\infty}{(q, ab, ac, ad, as, bc, bd, bs, cd, cs, ds, q\alpha; q)_\infty}\\
&\quad \times
{_8W_7}(\alpha; q/u, ab, ac, ad, as; q, \alpha uv/a^2).
\end{align*}
Setting $u=qr/abcds$ in the above equation and noting that $q\alpha=a^2bcds $, we find that
for $\max\{|a|, |b|, |c|, |d|, |s|\}<1, $
\begin{align*}
&\int_{0}^{\pi} \frac{h(\cos 2\theta; 1)h(\cos \theta; r)}{h(\cos \theta; a, b, c, d, s)} d\theta\\
&=\frac{ 2\pi (abcd, abcs, abds, acds, ar, r/a; q)_\infty}{(q, ab, ac, ad, as, bc, bd, bs, cd, cs, ds, a^2bcds; q)_\infty}\\
&\quad \times
{_8W_7}(a^2bcds/q; abcds/r, ab, ac, ad, as; q, r/a),
\end{align*}
which is the same as the integral formula in Theorem~\ref{nrintthma} if we interchanging $a$ and $s$.
\section{Strange evaluations of basic hypergeometric series}
Theorem \ref{liunewthmb} can be used to provide new proofs of some strange $q$-series identity.
We begin by proving  the following strange $q$-series identity
due to Andrews \cite[Eq. (4.5)]{Andrews79} (see also, \cite[Eq. (4.26)]{GesselStanton} and \cite[Eq. (4.5d)]{Chu1994}).
\begin{prop}\label{andrewspp1} {(\rm Andrews)}  We have the summation formula
\begin{align*}
&{_5\phi_4} \left({{q^{-n}, \alpha q^n, \alpha^{1/3}q^{1/3}, \alpha^{1/3}q^{2/3}, \alpha^{1/3} q}
\atop {\alpha^{1/2}q, -\alpha^{1/2}q, \alpha^{1/2}q^{1/2}, -\alpha^{1/2}q^{1/2}}} ;  q, q \right)\\
&=\frac{(1-\alpha)(1-\alpha^{1/3}q^{2n/3})(q; q)_n (\alpha^{1/3}; q^{1/3})_n (q\alpha)^{n/3}}
{(1-\alpha^{1/3})(1-\alpha q^{2n})(\alpha; q)_n (q^{1/3}; q^{1/3})_n}.
\nonumber
\end{align*}
\end{prop}
\begin{proof} Let $\omega$ be the primitive cube root of unity given by $\omega=\exp ({2\pi i}/{3})$. Then we have
\[
(1-x)(1-x\omega)(1-x\omega^2)=1-x^3.
\]
If we replace $(q, \alpha, a, b, c)$ by $(q^{1/3}, \alpha^{1/3}, a^{1/3},a^{1/3}\omega, a^{1/3}\omega^2)$
in Theorem \ref{rogersthm} and then use the above identity in the resulting equation,  we find that
 \begin{align}
 &\frac{(\alpha a^2/q; q)_\infty (\alpha^{1/3}q^{1/3}; q^{1/3})_\infty}
 {(\alpha a; q)_\infty (\alpha^{1/3} a/q^{2/3}; q^{1/3})_\infty}\label{euler:eqn1}\\
 &=\sum_{n=0}^\infty \frac{(1-\alpha^{1/3}  q^{2n/3}) (\alpha^{1/3}; q^{1/3})_n (q/a; q)_n
 (a/q)^n (q\alpha)^{n/3}}{(1-\alpha^{1/3})(\alpha a; q)_n (q^{1/3}; q^{1/3})_n}.\nonumber
 \end{align}
 On the other hand, taking $m=3$ in Theorem \ref{liunewthmb}, and then letting
 $(\alpha b, \alpha b_1, \alpha b_2, \alpha b_3 )= (\alpha^{1/2}q, -\alpha^{1/2}q, \alpha^{1/2}q^{1/2}, -\alpha^{1/2}q^{1/2})$
 and $(\alpha c_1, \alpha c_2, \alpha c_3)=(\alpha^{1/3}q^{1/3}, \alpha^{1/3}q^{2/3}, \alpha^{1/3} q),$
 we are led to derive the identity
  \begin{align*}
 &\frac{(\alpha a^2/q; q)_\infty (\alpha^{1/3}q^{1/3}; q^{1/3})_\infty}
 {(\alpha a; q)_\infty (\alpha^{1/3} a/q^{2/3}; q^{1/3})_\infty}\label{euler:eqn2}\\
&=\sum_{n=0}^\infty  \frac{(1-\alpha q^{2n}) (\alpha, q/a; q)_n (a/q)^n}
{(1-\alpha)(q, \alpha a; q)_n}{_5\phi_4} \left({{q^{-n}, \alpha q^n, \alpha^{1/3}q^{1/3}, \alpha^{1/3}q^{2/3}, \alpha^{1/3} q}
\atop {\alpha^{1/2}q, -\alpha^{1/2}q, \alpha^{1/2}q^{1/2}, -\alpha^{1/2}q^{1/2}}} ;  q, q \right),
 \end{align*}
 Combining the above two equations, we conclude that
 \begin{align*}
 &\sum_{n=0}^\infty \frac{(1-\alpha^{1/3}  q^{2n/3}) (\alpha^{1/3}; q^{1/3})_n (q/a; q)_n
 (a/q)^n (q\alpha)^{n/3}}{(1-\alpha^{1/3})(\alpha a; q)_n (q^{1/3}; q^{1/3})_n}\\
 &=\sum_{n=0}^\infty  \frac{(1-\alpha q^{2n}) (\alpha, q/a; q)_n (a/q)^n}
{(1-\alpha)(q, \alpha a; q)_n}{_5\phi_4} \left({{q^{-n}, \alpha q^n, \alpha^{1/3}q^{1/3}, \alpha^{1/3}q^{2/3}, \alpha^{1/3} q}
\atop {\alpha^{1/2}q, -\alpha^{1/2}q, \alpha^{1/2}q^{1/2}, -\alpha^{1/2}q^{1/2}}} ;  q, q \right).
 \end{align*}
 Appealing to Theorem \ref{liuthm1}, we can equate the coefficients of $a^n(q/a; q)_n/(\alpha a; q)_n$ 
 on both sides of the above equation to complete the proof
 of Proposition \ref{andrewspp1}.
\end{proof}
 Replacing $\alpha$ by $\alpha^3$ and $q$ by $q^3$ in (\ref{euler:eqn1}) and then setting $\alpha=1$ and $a=-q^3,$
 we deduce that (see, for example, \cite[Eq. (13)]{Andrews+Lewis+Liu})
\begin{equation}
\phi(-q)\phi(-q^3)=\frac{(q; q)_\infty(q^3; q^3)_\infty}
{(-q; q)_\infty (-q^3; q^3)_\infty}
=1+2\sum_{n=1}^\infty (-1)^n \frac{q^n(1+q^n)}{1+q^{3n}}.
\label{euler:eqn2}
\end{equation}
\begin{prop}\label{andrewspp2} {(\rm Andrews)}  We have the summation formula
\begin{align*}
&{_5\phi_4} \left({{q^{-n}, \alpha q^n, \alpha^{1/3}, \alpha^{1/3} e^{2\pi i/3}, \alpha^{1/3} e^{4\pi i/3}}
\atop {\sqrt{\alpha}, -\sqrt{\alpha}, \sqrt{q \alpha}, -\sqrt{q \alpha}}} ;  q, q \right)
\label{andrews:eqn2}\\
&=\begin{cases} 0 &\text {if $n \not \equiv  0 \pmod 3$}\\
\frac{(\alpha; q^3)_l (q; q)_{3l}\alpha^l}{(\alpha; q)_{3l}(q^3; q^3)_l}, & \text{if $n=3l$ }.
\end{cases}
\nonumber
\end{align*}
\end{prop}
This identity was first proved by Andrews \cite[Eq. (4.7)]{Andrews79}. For other proofs,
see \cite[Eq. (4.4d)]{Chu1994} and \cite[Eq. (4.32)]{GesselStanton}.
\begin{proof} If we first replace $q$ by $q^3$ in Theorem \ref{rogersthm} and then setting $b=qa$ and $c=q^2a$ in
the resulting equation, we obtain
\begin{equation*}
\sum_{n=0}^\infty \frac{(1-\alpha q^{6n}) (q/a; q)_{3n} (\alpha; q^3)_n \alpha^n (a/q)^{3n}}{(\alpha a; q)_{3n}(q^3; q^3)_n}
=\frac{(\alpha a^2/q^2; q)_\infty (\alpha; q^3)_\infty}{(\alpha a, q)_\infty (\alpha a^3/q^3; q^3)_\infty}.
\end{equation*}
Using the same argument that we used to prove Proposition \ref{andrewspp1},  from Theorem \ref{liunewthmb} we can deduce that
\begin{align*}
&\frac{(\alpha a^2/q^2; q)_\infty (\alpha; q^3)_\infty}{(\alpha a, q)_\infty (\alpha a^3/q^3; q^3)_\infty}\\
&=\sum_{n=0}^\infty  \frac{(1-\alpha q^{2n}) (\alpha, q/a; q)_n (a/q)^n}
{(q, \alpha a; q)_n}{_5\phi_4} \left({{q^{-n}, \alpha q^n, \alpha^{1/3}, \alpha^{1/3} e^{2\pi i/3}, \alpha^{1/3} e^{4\pi i/3}}
\atop {\sqrt{\alpha}, -\sqrt{\alpha}, \sqrt{q \alpha}, -\sqrt{q \alpha}}} ;  q, q \right).
\end{align*}
Combining the above two equations, we arrive at the $q$-identity
\begin{align*}
&\sum_{n=0}^\infty \frac{(1-\alpha q^{6n}) (q/a; q)_{3n} (\alpha; q^3)_n \alpha^n (a/q)^{3n}}{(\alpha a; q)_{3n}(q^3; q^3)_n}\\
&=\sum_{n=0}^\infty  \frac{(1-\alpha q^{2n}) (\alpha, q/a; q)_n (a/q)^n}
{(1-\alpha)(q, \alpha a; q)_n}{_5\phi_4} \left({{q^{-n}, \alpha q^n, \alpha^{1/3}, \alpha^{1/3} e^{2\pi i/3}, \alpha^{1/3} e^{4\pi i/3}}
\atop {\sqrt{\alpha}, -\sqrt{\alpha}, \sqrt{q \alpha}, -\sqrt{q \alpha}}} ;  q, q \right).
\end{align*}
Using Theorem \ref{liuthm1}, we compare the coefficients of $a^n(q/a; q)_n/(\alpha a; q)_n$  to complete the proof
of Proposition \ref{andrewspp2}.
\end{proof}
The $q$-Watson formula due to Andrews \cite {Andrews76} can be sated in the following proposition.
\begin{prop} \label{andrewspp3} {(\rm Andrews)}We have the $q$-formula
\begin{equation*}
{_4\phi_3} \left({{q^{-n}, \alpha q^n, \sqrt{\lambda}, -\sqrt{\lambda}} \atop {\sqrt{q \alpha}, -\sqrt{q \alpha}, \lambda}} ;  q, q \right)
=\begin{cases} 0 &\text {if $n$ is odd}\\
\frac{(q, \alpha q/\lambda; q^2)_{n/2} \lambda^{n/2}}{(q\alpha, q\lambda; q^2)_{n/2}}, & \text{if $n$ is even}.
\end{cases}
\end{equation*}
\end{prop}
\begin{proof} Replacing $q$ by $q^2$ in Theorem~\ref{rogersthm} and then setting $b=aq$ and $c=q\lambda/\alpha,$ we
find that
\begin{equation*}
\frac{(\alpha q, \lambda a/q; q)_\infty (\lambda, \alpha a^2/q; q^2)_\infty}
{(\alpha a, \lambda; q)_\infty (q\alpha, \lambda a^2/q; q^2)_\infty}=
\sum_{n=0}^\infty \frac{(1-\alpha q^{4n})(\alpha, q/a; q)_{2n} (q, q\alpha/\lambda; q^2)_n (a/q)^{2n}\lambda^n}
{(1-\alpha)(q, \alpha a; q)_{2n} (q\alpha, q\lambda; q^2)_n}
\end{equation*}
Taking $m=2$ and $(\alpha c_1, \alpha c_2, \alpha b_1, \alpha b_2, \alpha b )
=(\sqrt{\lambda}, -\sqrt{\lambda}, \sqrt{q\alpha}, -\sqrt{q\alpha}, \lambda)$ in Theorem \ref{liunewthmb},
we find that the left-hand side member of the above equation also equals
\begin{align*}
\sum_{n=0}^\infty  \frac{(1-\alpha q^{2n}) (\alpha, q/a; q)_n (a/q)^n}
{(1-\alpha)(q, \alpha a; q)_n} {_4\phi_3}
 \left({{q^{-n}, \alpha q^n, \sqrt{\lambda}, -\sqrt{\lambda}} \atop {\sqrt{q \alpha}, -\sqrt{q \alpha}, \lambda}} ;  q, q \right).
\end{align*}
Thus we have
\begin{align*}
&\sum_{n=0}^\infty  \frac{(1-\alpha q^{2n}) (\alpha, q/a; q)_n (a/q)^n}
{(1-\alpha)(q, \alpha a; q)_n} {_4\phi_3}
 \left({{q^{-n}, \alpha q^n, \sqrt{\lambda}, -\sqrt{\lambda}} \atop {\sqrt{q \alpha}, -\sqrt{q \alpha}, \lambda}} ;  q, q \right)\\
 &=\sum_{n=0}^\infty \frac{(1-\alpha q^{4n})(\alpha, q/a; q)_{2n} (q, q\alpha/\lambda; q^2)_n (a/q)^{2n}\lambda^n}
{(1-\alpha)(q, \alpha a; q)_{2n} (q\alpha, q\lambda; q^2)_n}.
\end{align*}
Using Theorem \ref{liuthm1}, we can equate the coefficients of $a^n(q/a; q)_n/(\alpha a; q)_n$ on both sides of  the above equation,
to complete the proof of Proposition \ref{andrewspp3}.
\end{proof}
Verma and Jain \cite[Eq. (5.4)]{VermaJain}  proved the following series summation  formula.
\begin{prop} \label{andrewspp4} {\rm (Verma and Jain )} We have the summation formula
\begin{equation*}
{_4\phi_3} \left({{q^{-2n}, \alpha^2 q^{2n}, \lambda, q\lambda} \atop { q\alpha, q^2 \alpha, \lambda^2}} ;  q^2, q^2 \right)
=\frac{\lambda^n (-q, q\alpha /\lambda; q)_n (1-\alpha)}{(\alpha, -\lambda; q)_n (1-\alpha q^{2n})}.
\end{equation*}
\end{prop}
\begin{proof} Replacing $\alpha$ by $-\alpha$ and then $( a, b, c)$ by $(\sqrt{a}, -\sqrt{a},\lambda/\alpha )$
in Theorem \ref{rogersthm}, we deduce  that
\[
\frac{(\lambda^2 a/q^2; q^2)_\infty (-\alpha, \alpha a/q; q)_\infty}
{(\alpha^2 a; q^2)_\infty (-\lambda, \lambda a/q^2; q)_\infty}
=\sum_{n=0}^\infty \frac{(1+\alpha q^{2n})(q^2/a; q^2)_n (q\alpha/\lambda, -\alpha; q)_n (\lambda a/q^2)^n}
{(\alpha^2 a; q^2)_n (q, -\lambda; q)_n}.
\]
Taking $m=2$ in Theorem \ref{liunewthmb}, then replacing $q$ by $q^2$ and $\alpha$ by $\alpha^2$,
and finally replacing $(\alpha^2 c_1,  \alpha^2 c_2, \alpha^2 b, \alpha^2 b_1,  \alpha^2 b_2)$ by
$(\lambda, q\lambda, \lambda^2,  q\alpha, q^2\alpha)$
we find that
 \begin{align*}
&\frac{(\lambda^2 a/q^2; q^2)_\infty (-\alpha, \alpha a/q; q)_\infty}
{(\alpha^2 a; q^2)_\infty (-\lambda, \lambda a/q^2; q)_\infty} \\
&=\sum_{n=0}^\infty  \frac{(1-\alpha^2 q^{4n}) (\alpha^2, q^2/a; q^2)_n (a/q^2)^n}
{(1-\alpha)(q^2, \alpha^2 a; q^2)_n}
{_{4}\phi_{3}} \left({{q^{-2n}, \alpha^2 q^{2n}, \lambda, q\lambda} \atop { q\alpha, q^2 \alpha, \lambda^2}} ;  q^2, q^2 \right).
 \end{align*}
 Equating the coefficients of $a^n(q^2/a; q^2)_n/(\alpha^2 a; q^2)_n$ on the right-hand side of the above two equations,
 we complete the proof of Proposition \ref{andrewspp4}.
\end{proof}


\section{Acknowledgments}
I am  grateful to the referee for many very helpful comments and suggestions.

\end{document}